\documentclass[oneside,12pt]{amsart}
\usepackage{amssymb, amscd}
\usepackage[all]{xy}
\usepackage{amsthm}
\usepackage{mathtools}
\usepackage{tikz}
\usepackage{enumerate}
\usepackage{multirow}
\usepackage{longtable}

\newtheorem*{thmm}{Main Theorem}

\newtheorem{thm}{Theorem}[section]
\newtheorem{cor}[thm]{Corollary}
\newtheorem{lem}[thm]{Lemma}

\newtheorem{prop}[thm]{Proposition}

\newtheorem{rem}[thm]{Remark}

\newcommand{\del}[2]{{}}

\newcommand{\Z}{\mathbb Z}
\newcommand{\G}{\Gamma}
\newcommand{\La}{\Lambda}
\newcommand{\R}{\mathbb R}

\title[]{Spin Structures on almost-flat manifolds}

\author{A.~G\c{a}sior}
\address{Maria Curie-Sk{\l}odowska University,  Lublin, Poland}%
\email{anna.gasior@poczta.umcs.lublin.pl}%

\author{N.~Petrosyan}
\address{Department of Mathematics, University of Southampton, Southampton, UK}%
\email{N.Petrosyan@soton.ac.uk}%

\author{A.~Szczepa\'{n}ski}
\address{Institute of Mathematics, University of Gda\'{n}sk, Gda\'{n}sk, Poland}%
\email{aszczepa@mat.ug.edu.pl}%

\thanks{}%

\date{\today}

\begin{document}
\maketitle
\begin{abstract} We give a necessary and sufficient  condition for almost-flat manifolds with cyclic holonomy to admit a Spin structure. Using this condition we find all $4$-dimensional orientable
almost-flat manifolds with cyclic holonomy that do not admit a Spin structure. 
\end{abstract}

\section{Introduction}
An {\it almost-flat} manifold is a closed manifold $M$ with the property that for any $\varepsilon >0$ there exists a Riemannian metric $g_{\varepsilon}$ on $M$ such that $|K_{\varepsilon}|\mathrm{diam}(M, g_{\varepsilon})^2<\varepsilon$ where $K_{\varepsilon}$ is the sectional curvature and  $\mathrm{diam}(M, g_{\varepsilon})$ is the diameter of $M$. According to a result of Gromov (see \cite{gromov})  every almost-flat manifold is finitely covered by a nilmanifold, i.e. a quotient of a connected, simply-connected nilpotent Lie group by a uniform lattice. Ruh (see \cite{ruh}) later improved Gromov's result by showing that in fact every almost-flat manifold is infra-nil. Let us recall the definition of such a manifold. 

Given   a connected and simply-connected nilpotent Lie group $N$,  the group of affine transformations of $N$ is defined as ${\rm Aff}(N)= N\rtimes {\rm Aut}(N)$. This group  acts on $N$ by 
$$(n,\varphi)\cdot m= n\varphi(m) \;\;\; \forall m, n\in N \mbox{ and } \forall \varphi \in {\rm Aut}(N).$$
\noindent Let $C$ be a maximal compact subgroup of ${\rm Aut}(N)$ and consider the subgroup $N\rtimes C$ of ${\rm Aff}(N)$. A discrete subgroup $\G\subset N\rtimes C$ that acts  co-compactly on $N$ 
is called an {\it almost-crystallographic group}.  In addition, if $\G$ is torsion-free then it is said to be {\it almost-Bieberbach}. In this case, the quotient $N/\G$ is a closed manifold called  {\it infra-nilmanifold} (modelled on $N$).

Conversely, every infra-nilmanifold is almost-flat because it is finitely covered by a nilmanifold and every nilmanifold is almost-flat (see \cite{gromov} and also \cite{BK}). So from now on we will make no distinction between the classes of almost-flat  and infra-nil manifolds. 

In this paper we study the problem of determining the existence of Spin structures on almost-flat manifolds. The existence of Spin structures  on flat manifolds and related invariants have  been investigated by the third author and others for the special case of flat manifolds (see for example \cite{DSS}, \cite{GS},  \cite{HS}, \cite{MP}, \cite{MP2},   \cite{PS}, and \cite{szcz}). Our results represent the first modest step towards understanding  this problem in the more general setting of almost-flat manifolds.

Before stating our main result let us recall the definition of a Spin structure on a smooth orientable manifold. We denote by $\mathrm{SO}(n)$ the real  special orthogonal group of rank $n$ and by $\mathrm{Spin}(n)$ its universal covering group. We also write $\lambda_n:\mathrm{Spin}(n)\to \mathrm{SO}(n)$ for the (double) covering homomorphism. A {\it Spin structure} on a smooth orientable manifold $M$ is an equivariant lift of its orthonormal frame bundle via the covering $\lambda_n$.  The existence of such a  lift is equivalent to the existence of a lift $\tilde{\tau}:M\to B\mathrm{Spin}(n)$ of the classifying map of the tangent bundle $\tau:M\to B\mathrm{SO}(n)$ 
such that $B_{\lambda_n}\circ\tilde{\tau}={\tau}$.  Equivalently, $M$ has a Spin structure if and only if  the second Stiefel-Whitney class $w_2(T M)$ vanishes (see \cite[p.~33-34]{kirby}).  
 
Returning to infra-nilmanifolds, by a classical result of Auslander (see \cite{auslander}) every almost-crystallographic subgroup $\G \subset \mathrm{Aff}(N)$ fits into an extension
$$1\to \La \to \G {\buildrel q\over \to} F \to 1$$ where $\Lambda=\G\cap N$ is a uniform lattice in $N$ and $F$ is a finite subgroup of $C$ called the {\it holonomy group} of the corresponding infra-nilmanifold $N/\G$. The conjugation action of $\Gamma$ on $\Lambda$ induces an action of the holonomy group $F$ on the factor groups of the adapted lower central series (see page 5) of the nilpotent lattice $\Lambda$.  This give us a representation $\theta:F\hookrightarrow \mathrm{GL}(n, \Z)$ where $n$ is the dimension of $N$.

Our main result is the following.
\begin{thmm}\label{main} Let $M$ be an almost-flat manifold  with holonomy group $F$.  Then $M$ is orientable if and only if $\det \theta =1$. Suppose $M$ is orientable and  a 2-Sylow subgroup of $F$ is cyclic, i.e.~$C_{2^m}=\langle t \;|\; t^{2^m}=1 \rangle$ for some $m\geq 0$.  Let $\Gamma_{ab}$ denote the abelianisation of the fundamental group $\G$ of $M$.
%Denote $\Gamma_2={q^{-1}(C_2)}$ and let $M(2)= N/{\Gamma_2}$. 
\begin{itemize}
\smallskip
\item[(a)]  If ${1\over 2}(n-\mathrm{Trace}(\theta(t)^{2^{m-1}}))\not\equiv 2  \; (\mathrm{ mod } \; 4),$ then $M$ has a Spin structure.
\medskip
\item[(b)] If ${1\over 2}(n-\mathrm{Trace}(\theta(t)^{2^{m-1}}))\equiv 2  \; (\mathrm{ mod } \; 4)$, then $M$ has a Spin structure if and only if the epimorphism ${q}_*:\G_{ab}\to C_{2^m}$ resulting from projection $q:\Gamma\to C_{2^m}$ factors through a cyclic group of order $2^{m+1}$.  
\end{itemize}
\end{thmm}

The conditions arising in the theorem are quite practical to check given a finite presentation of the fundamental group of the almost-flat manifold, i.e.~the associated almost-Bieberbach group. We illustrate this by finding all $4$-dimensional almost-flat manifolds whose holonomy group has a cyclic 2-Sylow subgroup that do not admit a Spin  structure. 

%\begin{quest}\rm
%Do all orientable almost-flat manifolds with holonomy $C_2$ modelled on a free, connected, simply connected, nilpotent Lie group $N$ admit a Spin structure?
%\end{quest}
%
%A Lie group is said to be {\it free nilpotent of class $c$} if its Lie algebra is free nilpotent of class %$c$. Let $M$ be an infra-nilmanifold modelled  on a  free $c$-step nilpotent Lie group with fundamental %group $\G$. 

%We denote by $Q$ the quotient group $\G/[\La, \La]$ and note that it fits into extension
%$$1\to \La_{ab} \to Q\to  F \to 1$$ where the abelianization $\La_{ab}$ of $\La$ is a free abelian group.  Recall that there is an induced action of $F$  on the abelianization of $\La$ induced by conjugation in $Q$.
%
%Throughout $C_m$ will denote a cyclic group of order $m$ and the $\Z_2$-rank (the rank of a maximal elementary abelian $2$-subgroup)  of an abelian group  $A$ will be denoted by $rk_2A$.
%
%
%\begin{thm} Let $M$ be an orientable infra-nilmanifold with  holonomy group $C_2$.  If $Q$ is torsion-free, then $M$ is Spin. 
%\end{thm}
%
%
%\begin{ques}\rm
%Can we  generalize these results to any cyclic holonomy which is a product of distinct primes?
%\end{ques}
%
%
%\begin{ques}\rm
%Can we  remove the torsion-free assumption? 
%\end{ques}

\section{Results}

We first show that the classifying map of the tangent bundle of an almost-flat manifold $M$ factors through the classifying space of the holonomy group $F$ and is  induced by a representation  $\rho:F\to  \mathrm{O}(n)$.  Let us  describe this representation.  

Define $\mathfrak{n}$ to be the Lie algebra corresponding to the nilpotent Lie group $N$ modelling $M$. Since $N$ is a connected and simply-connected, nilpotent Lie group, the differential defines an isomorphism $d:\mathrm{Aut}(N)\to \mathrm{Aut}(\mathfrak{n})$. Choose an inner product $\langle \;,\; \rangle$  on  $\mathfrak{n}$.  Since $d(C)$ is a compact subgroup of $\mathrm{Aut}(\mathfrak{n})$, we can define a new inner product  $\langle\langle \;,\; \rangle\rangle$  on $\mathfrak{n}$ that is also invariant under the action of $d(C)$ by letting
$$\langle\langle v, w \rangle\rangle=\int_{d(C)} \langle xv, xw \rangle \mu(x)\;\;\; \forall v, w \in \mathfrak{n},$$
where $\mu$ is a left-invariant Haar measure on $d(C)$.

 Now, we select  basis on $\mathfrak{n}$ orthonormal with respect to the new inner product.  Identifying this basis with the standard basis in $\R^n$ defines a vector space isomorphism $\eta:\mathfrak{n}\to \R^n$ and  a monomorphism $\delta:\mathrm{Aut}(\mathfrak{n})\to \mathrm{GL}(n)$ such that $\delta\circ d(C)\subseteq \mathrm{O}(n)$.   We define $\rho:F\hookrightarrow  \mathrm{O}(n)$ by restricting the domain and the codomain  of the composite homomorphism $$C\hookrightarrow \mathrm{Aut}(N)\xrightarrow{d} \mathrm{Aut}(\mathfrak{n})\xrightarrow{\delta}\mathrm{GL}(n)$$  to $F$ and $\mathrm{O}(n)$, respectively. It is crucial to note that  $\rho$ is well-defined up to isomorphism of representations. That is, for a different choice of the inner product and  the orthonormal basis  on $\mathfrak{n}$  one obtains a representation that is isomorphic to $\rho:F\hookrightarrow  \mathrm{O}(n)$. 

%Since $M$ is infra-nilmanfold modelled on the Lie group $N$, it can be given a Riemannian metric that invariant under the action of $\mathrm{Aff}(N)$. This is done by chosen an inner product on the tangent space of $N$ at the identity, i.e. on the nilpotent Lie algebra $\mathfrak{n}$ corresponding to $N$ and then extending this inner product to the whole tangent bundle $TN$ by left translations. This gives a metric on the Lie group $N$ that is invariant under the left action of $N$ on $N$. Since the group $C$ is a compact subgroup of $\mathrm{Aut}(N)$, we can then average this metric over $C$ to obtain a new metric on $N$ that is also invariant under the action of $C$ on $N$. In all, this new metric will be invariant under the action of $N\rtimes C$ and hence under the action of $\Gamma$. We can then given $M=N/\Gamma$ the induced metric which w

\begin{prop} \label{reduce}
Let $M$ be a $n$-dimensional almost-flat manifold modelled on a connected and simply connected nilpotent Lie group $N$. Denote by $\G$ the fundamental group of $M$ and let
$$1\to \La \to \G {\buildrel q\over \to} F \to 1$$  be the standard extension of $\G$. Then the classifying map $\tau:M\to {BO}(n)$ of the tangent bundle of $M$ factors through $BF$ and is induced by a composite homomorphism $\rho\circ q:\G\to F {\buildrel \rho\over \to} \mathrm{O}(n)$.
\end{prop}

\begin{proof} Let $\rho:F\hookrightarrow  \mathrm{O}(n)$ be the representation constructed above. This yields a map on the classifying spaces ${B}_{\rho}: BF\to BO(n)$ which is well-defined up to homotopy. Denote by $\sigma$ the pullback of the universal $n$-dimensional vector bundle on $BO(n)$ under the map ${B}_{\rho}$. Its  total space is the Borel construction  $EF \times_F \R^n$, i.e. the quotient of $EF \times \R^n$ by the action of $F$ given by $f\cdot(x, v)=(fx, \rho(f)v)$, $\forall f\in F$, and $\forall (x,v)\in  EF\times \R^n$.

We claim that the pullback bundle ${B}_q^*(\sigma)$ of $\sigma$ under the  map ${B}_q:B\Gamma \to BF$ is isomorphic to tangent bundle $T M\to M$. To see this, let $L_g:N\to N: h\mapsto gh$ be the left multiplication by an element $g$ in $N$. It is a standard fact from Lie groups
that the map $$\phi:TN\to N\times \mathfrak{n}: (g, v )\mapsto (g, d{L}_{g^{-1}}(v)),  \;\;\;\;\; \forall g\in N, \; \forall v\in \mathfrak{n},$$ gives a trivialisation of the tangent bundle of $N$. A quick computation shows that this map is equivariant with respect to the  action of $\Gamma$ on $N\times \mathfrak{n}$ given by
 $$\gamma \cdot(g, v)=(\gamma g, q(\gamma) (v)), \; \forall \gamma \in \Gamma \mbox{ and } \forall (g,v)\in N \times \mathfrak{n}.$$

Hence, we obtain a commutative diagram:
\[   \xymatrix{   & TN\ar[d] ^{/\Gamma}\ar[r]^{\phi} & N\times \mathfrak{n}\ar[d]^{/\Gamma} \\
                 & TM \ar[r]^{\overline{\phi}}  & N\times_{\Gamma} \mathfrak{n}} \]
where the resulting map $\overline{\phi}:TM\to N\times_{\Gamma} \mathfrak{n}$  gives an isomorphism between the tangent bundle of $M$ and $\overline{pr}_1:N\times_{\Gamma} \mathfrak{n}\to M$.
But since $N$ is a model for $E\Gamma$, we also have  a commutative diagram:
\[   \xymatrix{   & N\times_{\Gamma}\mathfrak{n}\ar[d]^{\overline{pr}_1}  \ar[r]^{\psi} & EF\times_{F}\R^n \ar[d]^{\sigma}  \\
                 & M \ar[r]^{{B}_q}  & BF} \]
for $\psi:N\times_{\Gamma}\mathfrak{n}\to EF\times_{F}\R^n: \; \{ g, v\}\mapsto \{ {E}_q(g), \eta(v)\}$. This finishes the claim and the proposition follows.
\end{proof}

\begin{rem}\rm If the manifold $M$ is orientable, then in the statement of the proposition the structure group $\mathrm{O}(n)$ can be replaced by $\mathrm{SO}(n)$.
\end{rem}

With the previous notation, we define the {\it classifying representation} of an oriented almost-flat manifold $M$ to be the composite homomorphism $\rho \circ q:\G\to \mathrm{SO}(n)$. Recall that it is well-defined up to isomorphism of representations. 

\begin{cor}Let $M$ be an orientable almost-flat manifold of dimension $n$ with  
fundamental group $\Gamma$.
Then M has a Spin structure if and only if there exists a homomorphism
$\epsilon : \Gamma \to \mathrm{Spin}(n)$ such that $\lambda_n\circ\epsilon
= \rho\circ q.$
\end{cor} 
\begin{proof} The manifold $M$ has a Spin structure if and only if  the classifying map $\tau=B_{\rho\circ q}:M\to B\mathrm{O}(n)$ 
has a lift $\tilde{\tau}:M\to B\mathrm{Spin}(n)$ such that $B_{\lambda_n}\circ\tilde{\tau}={B}_{\rho\circ q}$. 
Since $M=B\G$,   a homomorphism $\epsilon : \Gamma \to \mathrm{Spin}(n)$ satisfying $\lambda_n\circ\epsilon
= \rho\circ q$  yields a map $B_{\epsilon}:M\to B\mathrm{Spin}(n)$  such that  $B_{\lambda_n}\circ B_{\epsilon}={B}_{\rho\circ q}$. Hence, $M$ has a Spin structure.

For the other direction, assume  $M$ has a Spin structure. Then $w_2(T M)=0$. But $w_2(T M)$ is the image of the generator of $\mathrm{H}^2(B\mathrm{SO}(n), \Z_2)=\Z_2$ under the homomorphism $B_{\rho\circ q}^*:\mathrm{H}^2(B\mathrm{SO}(n), \Z_2)\to \mathrm{H}^2(B\G, \Z_2)$. Identifying this homomorphism with $({\rho\circ q})^*:\mathrm{H}^2(\mathrm{SO}(n), \Z_2)\to \mathrm{H}^2(\G, \Z_2)$ we obtain that the image of the generator of  $\mathrm{H}^2(\mathrm{SO}(n), \Z_2)$ is zero.  Reinterpreting the statement using group extensions, gives us a commutative diagram
\[   \xymatrix{   & \Z_2\  \ar[r]^{} & \mathrm{Spin}(n)\ar[r]^{\lambda_n}  &\mathrm{SO}(n)  \\
                 & \Z_2\ar[u]_{id} \ar[r] & \widetilde{\G}\ar[u]_{\omega} \ar[r]^{\pi} &\G\ar@/^1.5 pc/[l]^{s}\ar[u]_{\rho\circ q}} \]
                 where $\pi\circ s=id_{\G}$. Setting $\epsilon = \omega \circ s$ we have $\lambda_n\circ\epsilon
= \rho\circ q$ as desired.
\end{proof}

Next, we will show that the  representation $\rho:F\hookrightarrow \mathrm{O}(n)$  is isomorphic in $\mathrm{GL}(n)$ to a representation that arises from the action of the holonomy group on the factor groups of a certain adapted lower central series of the nilpotent lattice $\Lambda$. This representation will turn out to be more suitable for applications. 

To this end,  we denote by $$\Lambda=\gamma_1(\Lambda)> \gamma_2(\Lambda)>\cdots >\gamma_{c+1}(\Lambda)=1,$$ 
the  lower central series of $\Lambda$, i.e. $\gamma_{i+1}(\Lambda)=[\Lambda, \gamma_i(\Lambda)]$ for $1\leq i\leq c$. By Lemma 1.2.6 of \cite{dekimpe}, we have that $\sqrt[\Lambda]{\gamma_i(\Lambda)}=\Lambda \cap {\gamma_i(N)}$. By Lemmas 1.1.2-3 of \cite{dekimpe}, the resulting {\it adapted lower central series}
$$\Lambda=\sqrt[\Lambda]{\gamma_1(\Lambda)}>  \sqrt[\Lambda]{\gamma_2(\Lambda)}>\cdots >\sqrt[\Lambda]{\gamma_{c+1}(\Lambda)}=1,$$ has torsion-free factor groups 
$$Z_i={\sqrt[\Lambda]{\gamma_i(\Lambda)}\over{\sqrt[\Lambda]{\gamma_{i+1}(\Lambda)}}},\;\;\; 1\leq i\leq c.$$ Thus, each $Z_i\cong \Z^{k_i}$  for some positive integer $k_i$.
Just as in the case when $\Lambda$ is abelian, conjugation in $\Gamma$ induces an action of the holonomy group $F$ on each factor group $Z_i$. This gives  a faithful representation 
$$\theta:F\hookrightarrow \mbox{GL}(k_1, \Z)\times \cdots \times  \mbox{GL}(k_c, \Z)\hookrightarrow \mathrm{GL}(n, \Z), \;\;\; k_1+\cdots + k_c=n.$$ The representation is indeed  faithful since its kernel is a finite unipotent group and is therefore trivial.
\begin{prop}\label{equiv} The representations $\theta\otimes \R:F\hookrightarrow \mathrm{GL}(n)$ and $\rho:F\hookrightarrow \mathrm{O}(n)\subset \mathrm{GL}(n)$ are isomorphic.
\end{prop}

\begin{proof}  Since $F$ is finite, it suffices to show that the two representations have equal characters (see \cite[30.14]{CR}). Let $\mathcal{C}:\mathrm{Aff}(N)\to \mathrm{Aut}(N)$ denote the homomorphism defined by the conjugation action of the group of affine transformations  on the normal subgroup $N$. Note that restricted to the standard subgroup $\mathrm{Aut}(N)$ of $\mathrm{Aff}(N)$, this is just the identity homomorphism.   Let $\mathrm{exp}:\mathfrak{n}\to N$ be  the exponential map.  Recall that for any homomorphism $\varphi: N\to N$ there is a commutative diagram
\[   \xymatrix{   & \mathfrak{n}\ar[d] ^{\mathrm{exp}}\ar[r]^{d\varphi} &  \mathfrak{n}\ar[d]^{\mathrm{exp}} \\
                 & N \ar[r]^{{\varphi}}  & N}\]
Moreover each subgroup $\gamma_i(N)$ in the lower central series of $N$ is characteristic in $N$ and one has $\mathrm{exp}(\gamma_i(\mathfrak{n}))=\gamma_i(N)$   (see \cite[Lemma 1.2.5]{dekimpe}). 

Now, we choose a Mal'cev basis for $\mathfrak{n}$ so that the  images of its elements under the exponential map generate the lattice $\Lambda$.  By construction, the subspaces $V_i=\eta(\gamma_i(\mathfrak{n})), 1\leq i\leq c$ give us a filtration
 $$0=V_{c+1}\subset V_c\subset  \dots \subset V_1 = \R^n$$ with $\dim V_{i}= k_i+\dots + k_c$   and each $V_i$ is  left invariant under the action by the image of the homomorphism $\delta:\mathrm{Aut}(\mathfrak{n})\to \mathrm{GL}(n)$.  
For each $1\leq i\leq c$,  this defines a  representation $$\delta_i:\mathrm{Aut}(\mathfrak{n})\to \mathrm{GL}(V_i/V_{i+1}).$$  Let $\tilde{\rho}_i:\Gamma\to \mbox{GL}(k_i)$ denote the composition 
$$\Gamma\hookrightarrow \mathrm{Aff}(N)\xrightarrow{\mathcal{C}}  \mathrm{Aut}(N)\xrightarrow{d} \mathrm{Aut}(\mathfrak{n})\xrightarrow{\delta_i}\mbox{GL}(V_i/V_{i+1}).$$ Since $\Lambda$ is in the kernel of $\tilde{\rho}_i$, it gives rises to the representation $\rho_i:F\to \mathrm{GL}(k_i)$. Since $\delta$ and $\delta_1\oplus \cdots \oplus \delta_c$ have equal characters, $\rho$ and $\rho_1\oplus \cdots \oplus \rho_{c}$ have equal characters.

On the other hand, the representation $\theta$ is isomorphic to $\theta_1\oplus \cdots \oplus \theta_c$ where $\theta_i:F\to \mathrm{GL}(k_i, \Z)$ is induced from $\tilde{\theta}_i:\Gamma \to \mathrm{GL}(k_i, \Z)$ and the latter is defined by 
$$\Gamma \xrightarrow{\mathcal{C}|_{\Gamma}}  \mathrm{Aut}(\Lambda)\to\mbox{GL}(k_1, \Z)\times \cdots \times  \mbox{GL}(k_c,\Z)\xrightarrow{pr_i}\mbox{GL}(k_i,\Z),$$
for each $1\leq i \leq c$. So, to finish the proof it suffices to show that $\tilde{\rho}_i$ and $\tilde{\theta}_i\otimes \R$ have equal characters for each $1\leq i\leq c$. 

By  taking a closer look at $\tilde{\rho}_i$, it is  not difficult to see that it is isomorphic to the composition
$$\Gamma \xrightarrow{\mathcal{C}|_{\Gamma}} \mathrm{Aut}(N)\rightarrow \mathrm{Aut}(\gamma_i(N)/{\gamma_{i+1}(N)})\xrightarrow{d} \mathrm{GL}(\gamma_i(\mathfrak{n})/\gamma_{i+1}(\mathfrak{n}))$$
where we identify $\gamma_i(\mathfrak{n})/\gamma_{i+1}(\mathfrak{n})$ and $V_i/V_{i+1}$ via the isomorphism $\eta:\mathfrak{n}\to \R^n$ and where the second homomorphism is the natural map arising from the action of the automorphism group of $N$ on the lower central series of $N$.

On the other hand, the representation $\tilde{\theta}_i$ can be defined by the composition
$$\Gamma \xrightarrow{\mathcal{C}|_{\Gamma}}  \mathrm{Aut}(\Lambda)\rightarrow \mathrm{Aut}\Big(\sqrt[\Lambda]{\gamma_i(\Lambda)}/{\sqrt[\Lambda]{\gamma_{i+1}(\Lambda)}}\Big)$$ 
where the second homomorphism is the natural map arising from the action of the automorphism group of $\Lambda$ on the adapted lower central series of $\Lambda$.

From the choice of the Mal'cev basis on $\mathfrak{n}$ and the fact that $\sqrt[\Lambda]{\gamma_i(\Lambda)}/{\sqrt[\Lambda]{\gamma_{i+1}(\Lambda)}}$ is a lattice in the Euclidean  group $\gamma_i(N)/{\gamma_{i+1}(N)}$, it follows that $\tilde{\theta}_i\otimes \R$ is isomorphic to the composition
$$\Gamma \xrightarrow{\mathcal{C}|_{\Gamma}} \mathrm{Aut}(N)\rightarrow \mathrm{Aut}(\gamma_i(N)/{\gamma_{i+1}(N)})$$ and hence to $\tilde{\rho}_i$. This finishes the proof.
\end{proof}

\begin{rem}\label{orient} \rm It follows  that the almost-flat manifold $M$ is orientable if and only if the image of the representation $\theta:F\hookrightarrow   \mathrm{GL}(n, \Z)$ lies inside  $\mathrm{SL}(n, \Z)$.
\end{rem}

\begin{lem}\label{sylow} Let $M$ be an orientable almost-flat manifold  with holonomy group $F$.  Let  $S$ be  a 2-Sylow subgroup of $F$ and set $M(2)= N/{q^{-1}(S)}$.  Then $M$ has a Spin structure if and only if $M(2)$ has a Spin structure.
\end{lem}
\begin{proof} Recall that the second Stiefel-Whitney class $w_2(T M)$ is the obstruction for the existence of a Spin structure on $M$. The  inclusion $i : q^{-1}(S) \hookrightarrow \G$  induces a homomorphism $i^* : \mathrm{H}^2(M, \Z_2) \to \mathrm{H}^2(M(2),\Z_2)$. This is a monomorphism because  $q^{-1}(S)$ is a subgroup of $\G$ of  odd index. Since $w_2(T M(2)) = i^*(w_2(T M)) $, we obtain that $w_2(T M) = 0$ if and only if $w_2(T M(2)) = 0$.
\end{proof}

We also need the following lemma that will help us determine  whether  almost-flat manifolds with cyclic holonomy group  admit  Spin structures.
\begin{lem}\label{GG} Let $A\in \mathrm{SO}(n)$ be of order $2^m$, $m>0$ and let $a\in \lambda_	n^{-1}(\langle A\rangle)$. Then $a$ is of order $2^{m+1}$ if and only if $${1\over 2}(n-\mathrm{Trace}(A^{2^{m-1}}))\equiv 2  \; (\mathrm{ mod } \; 4).$$
\end{lem}
\begin{proof}The case $m=1$ is a well-known (see \cite{griess}, \cite{GG}). The general case follows easily from this case where we replace the matrix $A$ with $A^{2^{m-1}}$.
\end{proof}

We are now ready to prove the main theorem of the paper.

\begin{proof}[Proof of  Main Theorem] For part (a), suppose  ${1\over 2}(n-\mathrm{Trace}(\theta(t)^{2^{m-1}}))\not\equiv 2  \; (\mathrm{ mod } \; 4)$. Then, by Lemma \ref{GG} and Proposition \ref{equiv},  we have $\lambda_n^{-1}(\rho(C_{2^m}))\cong C_2\times C_{2^m}$. So, the restriction $\lambda_n: \lambda_n^{-1}(\rho(C_{2^m}))\to \rho(C_{2^m})$ splits and hence the classifying homomorphism $\rho \circ q:\Gamma\to \mathrm{SO}(n)$ lifts to the universal covering group $\mathrm{Spin}(n)$ of $\mathrm{SO}(n)$. This, by Proposition \ref{reduce}, insures that $M$ has a Spin structure.

For part (b),  in view of Lemma \ref{sylow}, we can assume the $C_{2^m}$ is the whole holonomy group of $M$.  Thus, $M$ has a Spin structure if and only if  there is a lift  $l:\G\to \mbox{Spin}(n)$ of the composite homomorphism $\rho\circ {q}: \Gamma\xrightarrow{q} C_{2^m}\xrightarrow{\rho} \mathrm{SO}(n)$.  
%We can also assume that ${1\over 2}(n-\mathrm{Trace}(\theta(t)))\equiv 2  \; (\mathrm{ mod } \; 4)$, for otherwise, we would be done by part (a). 
But by our assumption and Lemma \ref{sylow}, the primage $\lambda_n^{-1}(\rho(C_2))$ is isomorphic to   $C_{2^{m+1}}$. This shows that there is lift  $l:\G\to \mbox{Spin}(n)$ if and only if  $q:\G\to C_{2^m}$ factors through $C_{2^{m+1}}$ which happens if and only if $q_*:\G_{ab}\to C_{2^m}$ factors through $C_{2^{m+1}}$. 
\end{proof}

%\begin{rem} \rm For the practical computations that will come up in the next section,    we note that ${q}_*:\G_{ab}\to C_{2^m}$ does not factors through a cyclic group of order $4$ if and only %if  the induced $\Z$-module epimorphism $\G_{ab}\otimes  \Z_{2^{m+1}}\to \Z_{2^m}$ splits.
%\end{rem}

\section{Examples}
It is well-known that all closed orientable manifolds of dimension at most $3$ have a Spin structure (see \cite[p.~35]{kirby}, \cite[Exercise 12.B and VII, Theorem 2]{MS}).  
Next we  give a  list of  $4$-dimensional orientable almost-flat manifolds modelled on a connected,  simply-connected nilpotent  Lie group  $N$ that cannot have a Spin structure.   This list is complete in the sense  that,  up to dimension $4$,  it gives all possible examples of   orientable almost-flat manifolds whose holonomy has a cyclic  $2$-Sylow subgroup not admitting a Spin structure (see \cite{PS}). In fact, we will see that in each of these examples the holonomy group is $C_2$. In contrast all flat manifolds with holonomy $C_2$ have  a Spin structure (see \cite[Theorem 3.1(3)]{HS}).

For this purpose, we use the  classification of the associated almost-Bieberbach groups given in Sections 7.2 and 7.3 of \cite{dekimpe}.
\subsection{N is 2-step nilpotent} The only family of almost-flat manifolds without a Spin structure are classified by number 5, $Q=C2$ on page 171 of \cite{dekimpe}.

For each integer $k>0$, the almost-Bieberbach group $\G_k$  has the presentation 
\begin{align*}
\Gamma_k =\langle  a, b, c, d, \alpha \;|\; & [b, a]=1, [c,a]=d^k, [d, a]=1, [c,b]=d^k, [d, b]=1, [d,c]=1, \\
& \alpha^2=d, \alpha a \alpha^{-1} =b^{-1}, \alpha b \alpha^{-1}= a^{-1}, \alpha d \alpha^{-1} = d, \alpha c\alpha^{-1}=c^{-1} \rangle
\end{align*}
where
$\Lambda = \langle a, b, c, d \rangle$ and $\sqrt[\Lambda]{[\Lambda, \Lambda]}= \langle d \rangle.$
Since the representation  $\theta:C_2\hookrightarrow   \mathrm{GL}(4, \Z)$ arises from the conjugation by the element $\alpha$ on $\Lambda$, it is generated  by the matrix
{\small{$$\displaystyle{\left[\begin{array}{rrrr} 
1 &0 & 0 & 0\\
0 &-1 & 0 & 0\\
0 &0  & 0 & -1\\
0&0 & -1 & 0\end{array}\right]}$$}}
 which lies in $\mathrm{SL}(4, \Z)$. So, by Remark \ref{orient}, $M_k$ is orientable for all $k>0$.

The  abelianization of $\G_k$ has the presentation
$$(\Gamma_k)_{ab} =\langle  \bar{a},  \bar{c}, \bar{\alpha} \;|\;  {\bar{\alpha}}^{2k}=1\rangle = C_{\infty}^2 \times C_{2k}.$$ The map $q_*:(\Gamma_k)_{ab} \to C_2$ can then be seen as  the epimorphism arising from projection of the $C_{2k}$-factor onto $C_2$. Therefore, it does not factor through $C_4$ if and only if $k$ is odd. So, by  Main Theorem, part (b), $M_k$ does not have a  Spin structure if and only if $k$ is odd.

\subsection{N is 3-step nilpotent}  In this case, we find 3 families of almost-flat manifolds without a Spin structure. 

The first family is classified by number 3, $Q=\langle (2l, 1)\rangle$ on page 220 of \cite{dekimpe}. For each $k,l>0$, the associated almost-Bieberbach group $\G_k$  has the presentation 
\begin{align*}
\Gamma_{k,l} =\langle  a, b, c, d, \alpha \;|\; & [b, a]=c^{2l}d^{(2l-1)k}, [c,a]=1, [d, a]=1, [c,b]=d^{2k}, [d, b]=1, [d,c]=1, \\
& \alpha^2=d, \alpha a  =a\alpha c, \alpha b =b^{-1}\alpha, \alpha d \alpha^{-1} = d, \alpha c\alpha^{-1}=c^{-1} \rangle
\end{align*}
where
$\Lambda = \langle a, b, c, d \rangle$,  $\sqrt[\Lambda]{[\Lambda, \Lambda]}= \langle c, d \rangle$ and  $\sqrt[\Lambda]{\gamma_3(\Lambda)}= \langle d \rangle.$
The representation  $\theta:C_2\hookrightarrow   \mathrm{GL}(4, \Z)$  is generated  by the matrix
{\small{$$\displaystyle{\left[\begin{array}{rrrr} 
1 &0 & 0 & 0\\
0 &-1 & 0 & 0\\
0 &0 & -1 & 0\\
0&0 & 0 & 1\end{array}\right]}$$}} which lies in $\mathrm{SL}(4, \Z)$. So, $M_{k,l}$ is orientable for all $k>0$.

The  abelianization of $\G_{k,l}$ has the presentation
$$(\Gamma_{k,l})_{ab} =\langle  \bar{a},  \bar{b},  \bar{\alpha} \;|\; {\bar{b}}^2=1, {\bar{\alpha}}^{2k}=1\rangle = C_{\infty}\times C_2 \times C_{2k}.$$ The map $q_*:(\Gamma_{k,l})_{ab} \to C_2$ is  the epimorphism arising from projection of the $C_{2k}$-factor onto $C_2$. Therefore, it does not factor through $C_4$ if and only if $k$ is odd. So, by Main Theorem, part (b), $M_{k,l}$ does not have a Spin structure if and only if $k$ is odd.

The second family is classified by number 5, $Q=\langle (2l, 0)\rangle$, on page 222 of \cite{dekimpe}. For each $k,l>0$, the associated almost-Bieberbach group  $\G_{k,l}$  has the presentation 
\begin{align*}
\Gamma_{k,l} =\langle  a, b, c, d, \alpha \;|\; & [b, a]=c^{2l}, [c,a]=d^k, [d, a]=1, [c,b]=d^{-k}, [d, b]=1, [d,c]=1, \\
& \alpha^2=d, \alpha a  =b\alpha, \alpha b =a\alpha, \alpha d \alpha^{-1} = d, \alpha c\alpha^{-1}=c^{-1} \rangle
\end{align*}
where
$\Lambda = \langle a, b, c, d \rangle$,  $\sqrt[\Lambda]{[\Lambda, \Lambda]}= \langle c, d \rangle$ and  $\sqrt[\Lambda]{\gamma_3(\Lambda)}= \langle d \rangle.$
The representation  $\theta:C_2\hookrightarrow   \mathrm{GL}(4, \Z)$  is generated  by the matrix
{\small{$$\displaystyle{\left[\begin{array}{rrrr} 
1 & 0 & 0 & 0\\
0 & -1 & 0 & 0\\
0 &0 & 0 & 1\\
0&0 & 1 & 0\end{array}\right]}$$}} which lies in $\mathrm{SL}(4, \Z)$. So, $M_{k,l}$ is orientable for all $k>0$.

The  abelianization of $\G_{k,l}$ has the presentation
$$(\Gamma_k)_{ab} =\langle  \bar{a},  \bar{c}, \bar{\alpha} \;|\;  {\bar{c}}^2=1, {\bar{\alpha}}^{2k}=1\rangle = C_{\infty}\times C_2 \times C_{2k}.$$ The map $q_*:(\Gamma_{k,l})_{ab} \to C_2$ is the epimorphism resulting from projection of the $C_{2k}$-factor onto $C_2$. Therefore, it does not factor through $C_4$ if and only if $k$ is odd. So, by Main Theorem, part (b), $M_{k,l}$ does not have a Spin structure if and only if $k$ is odd.

The third family is classified by number 5, $Q=\langle (2l+1, 0)\rangle$, on page 222 of \cite{dekimpe}. For each $k, l>0$, the associated almost-Bieberbach group $\G_{k,l}$  has the presentation 
\begin{align*}
\Gamma_{k,l} =\langle  a, b, c, d, \alpha \;|\; & [b, a]=c^{2l+1}, [c,a]=d^k, [d, a]=1, [c,b]=d^{-k}, [d, b]=1, [d,c]=1, \\
& \alpha^2=d, \alpha a  =b\alpha, \alpha b =a\alpha, \alpha d \alpha^{-1} = d, \alpha c\alpha^{-1}=c^{-1} \rangle
\end{align*}
where
$\Lambda = \langle a, b, c, d \rangle$,  $\sqrt[\Lambda]{[\Lambda, \Lambda]}= \langle c, d \rangle$ and  $\sqrt[\Lambda]{\gamma_3(\Lambda)}= \langle d \rangle.$
The representation  $\theta:C_2\hookrightarrow   \mathrm{GL}(4, \Z)$  is generated  by the matrix
{\small{$$\displaystyle{\left[\begin{array}{rrrr} 
1 & 0 & 0 & 0\\
0 & -1 & 0 & 0\\
0 &0 & 0 & 1\\
0&0 & 1 & 0\end{array}\right]}$$}}  which lies in $\mathrm{SL}(4, \Z)$. So, $M_{k,l}$ is orientable for all $k>0$.

The  abelianization of $\G_{k, l}$ has the presentation
$$(\Gamma_k)_{ab} =\langle  \bar{a}, \bar{\alpha} \;|\;  {\bar{\alpha}}^{2k}=1\rangle = C_{\infty} \times C_{2k}.$$ The map $q_*:(\Gamma_{k,l})_{ab} \to C_2$ can then be seen as  the epimorphism arising from projection of the $C_{2k}$-factor onto $C_2$. Therefore, it does not factor through $C_4$ if and only if $k$ is odd. So, by Main Theorem, part (b), $M_{k,l}$ does not have a Spin structure if and only if $k$ is odd.\\

We summarise our investigations in the table below.  Every $4$-dimensional almost-Bieberbach group $\Gamma$ fits into an extension
$$0\to \Z\to \Gamma \to Q \to 1$$ where $Q$ is a $3$-dimensional almost-crystallographic group (see \cite[\S 6.3]{dekimpe}).  If $N$ is $2$-step nilpotent, then $Q$ is in fact a crystallographic group.  The first column of the table indicates the number of the associated almost-crystallographic  group $Q$ as shown in \cite[\S 7.2-3]{dekimpe}  and the page number in \cite{dekimpe} where the presentation of $\G$ is given. The second column gives the classification of $Q$ as in the International Tables for Crystallography (I.T.) or as in  \cite[\S 7.1]{dekimpe}. The third column indicates the nilpotency class of the  group $N$ on which the almost-flat manifold is modelled. Columns four and five show the abelianization and the holonomy group, respectively. The last column indicates  the exact parameters for which the associated almost-flat manifold cannot admit a Spin structure.

{\small
\setlength\LTpost{0pt}
\addtolength{\tabcolsep}{-2pt}
\begin{longtable}{|c|c|c|c|c|c|c|c|}%[h]
	\caption*{\small Almost-flat manifolds without Spin structures}	
	\label{tab:homology} \\
%	\begin{changemargin}{-4.3cm}{1.0cm}
%\leftmargin{-1in}
%\begin{tabular}{|c|c|cl}
	\hline
	 \cite[\S7.2-3]{dekimpe} & $Q$ & class & $\G_{ab}$ & Holonomy &  Parameters   \\* \hline\hline
\multirow{1}{*}{5, p.~171}
&$C$2
		&$2$ & $C_{\infty}^2 \times C_{2k}$& $C_2$& $k$ odd 
		\\
	\hline
	\multirow{1}{*}{3, p.~220}
&$\langle (2l, 1)\rangle$
		& 3 & $C_{\infty}\times C_2 \times C_{2k}$&$C_2$& $l>0$, $k$ odd  
		\\
	\hline
	\multirow{1}{*}{5, p.~222}
&$\langle (2l, 0)\rangle$
		& 3 & $C_{\infty}\times C_2 \times C_{2k}$&$C_2$&  $l>0$, $k$ odd 
		\\
	\hline
	\multirow{1}{*}{5, p.~222}
&$\langle (2l+1, 0)\rangle$
		& 3& $C_{\infty} \times C_{2k}$ &$C_2$&  $l>0$, $k$ odd 
		\\
	\hline
\end{longtable}
%\begin{center}{\footnotesize
%\noindent Explain the table!!!
%}\end{center}
%\end{tabular}
%\end{changemargin}
}

\end{document}